\documentclass[oneside,reqno,american,english]{amsart}
\usepackage[latin9]{inputenc}
\usepackage{mathrsfs}
\usepackage{mathtools}
\usepackage{amsbsy}
\usepackage{amstext}
\usepackage{amsthm}
\usepackage{amssymb}
\usepackage{geometry}
\usepackage{setspace}

\makeatletter

\providecommand{\tabularnewline}{\\}

\theoremstyle{plain}
\newtheorem{prop}{\protect\propositionname}
\theoremstyle{plain}
\newtheorem{thm}{\protect\theoremname}
\theoremstyle{plain}
\newtheorem{cor}{\protect\corollaryname}

\@ifundefined{date}{}{\date{}}
\PassOptionsToPackage{reqno}{amsmath}

\usepackage{mathtools}

\makeatletter
\AtBeginDocument{\tagsleft@false} 
\makeatother

\usepackage{chngcntr}
\counterwithout{equation}{section}

\usepackage{microtype}   
\usepackage{setspace}    
\makeatother

\usepackage{babel}
\addto\captionsamerican{\renewcommand{\corollaryname}{Corollary}}
\addto\captionsamerican{\renewcommand{\propositionname}{Proposition}}
\addto\captionsamerican{\renewcommand{\theoremname}{Theorem}}
\addto\captionsenglish{\renewcommand{\corollaryname}{Corollary}}
\addto\captionsenglish{\renewcommand{\propositionname}{Proposition}}
\addto\captionsenglish{\renewcommand{\theoremname}{Theorem}}
\providecommand{\corollaryname}{Corollary}
\providecommand{\propositionname}{Proposition}
\providecommand{\theoremname}{Theorem}

\begin{document}
\title{Transcendence Results for $\boldsymbol{\Gamma^{\left(n\right)}\left(1\right)}$
and Related Sequences of Generalized Constants}
\author{Michael R. Powers}
\address{Department of Finance, School of Economics and Management, Tsinghua
University, Beijing, China 100084}
\email{powers@sem.tsinghua.edu.cn}
\date{April 12, 2026}
\begin{abstract}
Neither the Euler-Mascheroni constant, $\gamma=0.577215\ldots$, nor
the Euler-Gompertz constant, $\delta=0.596347\ldots$, is currently
known to be irrational. However, it has been proved that at least
one of them is transcendental. The two constants are related through
a well-known equation of Hardy, equivalent to $\gamma+\delta/e=\textrm{Ein}\left(1\right)$,
which recently has been generalized to $\gamma^{\left(n\right)}+\delta^{\left(n\right)}/e=\eta^{\left(n\right)},\:n\geq0$
for sequences of constants $\gamma^{\left(n\right)}$, $\delta^{\left(n\right)}$,
and $\eta^{\left(n\right)}$ (derived respectively from raw, conditional,
and partial moments of the $\textrm{Gumbel}\left(0,1\right)$ probability
distribution). Investigating $\gamma^{\left(n\right)}=\left(-1\right)^{n}\Gamma^{\left(n\right)}\left(1\right),\:n\geq1$
through $\textrm{Gumbel}\left(0,1\right)$ generating functions, we
find that $\gamma^{\left(2n\right)}\in\mathbb{Q}\left[\gamma,\gamma^{\left(2\right)},\gamma^{\left(3\right)},\ldots,\gamma^{\left(2n-1\right)}\right]$
for $n\geq2$ and $\gamma^{\left(n\right)}$ is transcendental infinitely
often. We then show, via a theorem of Shidlovskii, that the $\eta^{\left(n\right)}$
are algebraically independent, and therefore transcendental, for all
$n\geq0$, implying that at least one element of each pair, $\left\{ \gamma^{\left(n\right)},\delta^{\left(n\right)}/e\right\} $
and $\left\{ \gamma^{\left(n\right)},\delta^{\left(n\right)}\right\} $,
and at least two elements of the triple $\left\{ \gamma^{\left(n\right)},\delta^{\left(n\right)}/e,\delta^{\left(n\right)}\right\} $
are transcendental for all $n\geq1$. Further analysis of the $\gamma^{\left(n\right)}$
and $\eta^{\left(n\right)}$ reveals that both the $\delta^{\left(n\right)}/e$
and $\delta^{\left(n\right)}$ are transcendental infinitely often
with lower asymptotic densities of at least $1/2$. Finally, we provide
parallel results for the sequences $\widetilde{\delta}^{\left(n\right)}$
and $\widetilde{\eta}^{\left(n\right)}$ satisfying the ``non-alternating
analogue'' equation $\gamma^{\left(n\right)}+\widetilde{\delta}^{\left(n\right)}/e=\widetilde{\eta}^{\left(n\right)}$.
\end{abstract}

\maketitle

\section{Introduction}

\subsection{Three Fundamental Sequences}

\begin{singlespace}
\noindent\phantom{}\medskip{}

\noindent Let $\gamma=0.577215\ldots$ denote the Euler-Mascheroni
constant and $\delta=0.596347\ldots$ denote the Euler-Gompertz constant.
Although neither $\gamma$ nor $\delta$ has been shown to be irrational,
Aptekarev {[}1{]} was the first to note their disjunctive irrationality;
that is, at least one of the two numbers must be irrational. Interestingly,
Aptekarev's observation in {[}1{]} was based on earlier work of Shidlovskii
{[}2{]} and Mahler {[}3{]} that implied, but did not explicitly state,
the disjunctive irrationality of $\left\{ \gamma,\delta\right\} $.
More recently, Rivoal {[}4{]} strengthened this result to show the
disjunctive transcendence of $\left\{ \gamma,\delta\right\} $. Lagarias
{[}5{]} provides a comprehensive treatment of results involving these
two constants.
\end{singlespace}

\begin{singlespace}
The two constants are related by the intriguing equation
\begin{equation}
\delta=-e\left[\gamma-{\displaystyle \sum_{k=1}^{\infty}\dfrac{\left(-1\right)^{k+1}}{k\cdot k!}}\right],
\end{equation}
introduced by Hardy {[}6{]}, which can be rewritten as
\begin{equation}
\gamma+\dfrac{\delta}{e}=-{\displaystyle \sum_{k=1}^{\infty}\dfrac{\left(-1\right)^{k}}{k\cdot k!}}=\textrm{Ein}\left(1\right),
\end{equation}
where $\textrm{Ein}\left(z\right)=-{\textstyle \sum_{k=1}^{\infty}\left(-z\right)^{k}/\left(k\cdot k!\right)}$
is the entire exponential-integral function. Both (1) and (2) follow
immediately from setting $z=-1$ in the power-series expansion of
the exponential-integral function (using the principal branch for
$z\in\mathbb{R}_{<0}$), $\textrm{Ei}\left(z\right)=\gamma+\ln\left|z\right|+{\textstyle \sum_{k=1}^{\infty}}z^{k}/\left(k\cdot k!\right)$,
and recognizing that $\delta=-e\textrm{Ei}\left(-1\right)$. Although
it is known that $\gamma+\delta/e$ is transcendental, this fact is
infrequently mentioned in the literature. Aptekarev {[}1{]} and Lagarias
{[}5{]} (Section 3.16, p. 608) attributed the observation that $1-1/e$
and $-\left(\gamma+\delta/e\right)$ are algebraically independent,
and therefore transcendental, to a more general result asserted (without
proof) by Mahler {[}3{]}.\footnote{We note that Aptekarev {[}1{]} contains a typographical error, referring
to $-\left(\gamma-\delta/e\right)$ instead of $-\left(\gamma+\delta/e\right)$.} 

Recently, we provided the following probabilistic interpretation of
(2):\footnote{See Powers {[}7{]}.}
\begin{equation}
\textrm{E}_{X}\left[X\right]-\textrm{E}_{X}\left[X\mid X\leq0\right]\Pr\left\{ X\leq0\right\} =\textrm{E}_{X}\left[X^{+}\right],
\end{equation}
where $X\sim\textrm{Gumbel}\left(0,1\right)$ (with cumulative distribution
function $F_{X}\left(x\right)=\exp\left(-e^{-x}\right),\:x\in\mathbb{R}$)
and $X^{+}=\max\left\{ X,0\right\} $, so that
\[
\gamma=\textrm{E}_{X}\left[X\right]={\displaystyle \int_{-\infty}^{\infty}}x\exp\left(-x-e^{-x}\right)dx,
\]
\[
\delta=-\textrm{E}_{X}\left[X\mid X\leq0\right]=-e{\displaystyle \int_{-\infty}^{0}}x\exp\left(-x-e^{-x}\right)dx,
\]
\[
\dfrac{1}{e}=\Pr\left\{ X\leq0\right\} ={\displaystyle \int_{-\infty}^{0}}\exp\left(-x-e^{-x}\right)dx,
\]
and
\[
-\sum_{k=1}^{\infty}\dfrac{\left(-1\right)^{k}}{k\cdot k!}=\textrm{E}_{X}\left[X^{+}\right]={\displaystyle \int_{0}^{\infty}}x\exp\left(-x-e^{-x}\right)dx.
\]

The identity in (3) suggests the natural generalization
\begin{equation}
\textrm{E}_{X}\left[X^{n}\right]-\textrm{E}_{X}\left[X^{n}\mid X\leq0\right]\Pr\left\{ X\leq0\right\} =\textrm{E}_{X}\left[\left(X^{+}\right)^{n}\right],\quad n\in\mathbb{Z}_{\geq0}
\end{equation}
(with $\left(X^{+}\right)^{0}\coloneqq\mathbf{1}_{\left\{ X>0\right\} }$),
which can be expressed in the manner of (2) as
\begin{equation}
\gamma^{\left(n\right)}+\dfrac{\delta^{\left(n\right)}}{e}=-n!{\displaystyle \sum_{k=1}^{\infty}\dfrac{\left(-1\right)^{k}}{k^{n}\cdot k!}},
\end{equation}
where
\[
\gamma^{\left(n\right)}=\textrm{E}_{X}\left[X^{n}\right]={\displaystyle \int_{-\infty}^{\infty}}x^{n}\exp\left(-x-e^{-x}\right)dx,
\]
\begin{equation}
\delta^{\left(n\right)}=-\textrm{E}_{X}\left[X^{n}\mid X\leq0\right]=-e{\displaystyle \int_{-\infty}^{0}}x^{n}\exp\left(-x-e^{-x}\right)dx,
\end{equation}
and
\begin{equation}
-n!{\displaystyle \sum_{k=1}^{\infty}\dfrac{\left(-1\right)^{k}}{k^{n}\cdot k!}}=\textrm{E}_{X}\left[\left(X^{+}\right)^{n}\right]={\displaystyle \int_{0}^{\infty}}x^{n}\exp\left(-x-e^{-x}\right)dx.
\end{equation}

For ease of exposition, we will set
\begin{equation}
\eta^{\left(n\right)}=-n!{\displaystyle \sum_{k=1}^{\infty}\dfrac{\left(-1\right)^{k}}{k^{n}\cdot k!}}
\end{equation}
so that (5) may be rewritten as
\begin{equation}
\gamma^{\left(n\right)}+\dfrac{\delta^{\left(n\right)}}{e}=\eta^{\left(n\right)},
\end{equation}
and note that
\[
\gamma^{\left(n\right)}=\left(-1\right)^{n}\Gamma^{\left(n\right)}\left(1\right).
\]
Further, we will refer to the sequences $\gamma^{\left(n\right)}$,
$\delta^{\left(n\right)}$, and $\eta^{\left(n\right)}$, for $n\in\mathbb{Z}_{\geq0}$,
as the \emph{generalized Euler-Mascheroni}, \emph{Euler-Gompertz},
and \emph{Eta} constants, respectively (with $\gamma^{\left(1\right)}=\gamma$,
$\delta^{\left(1\right)}=\delta=-e\textrm{Ei}\left(-1\right)$, and
$\eta^{\left(1\right)}=\eta=\textrm{Ein}\left(1\right)$ denoting
the \emph{ordinary} Euler-Mascheroni, Euler-Gompertz, and Eta constants).
Note that the generalized Euler-Mascheroni constants, $\gamma^{\left(n\right)}$,
are indexed via superscripts to distinguish them from the well-known
sequence of Stieltjes constants, commonly denoted by $\gamma_{n}$.
Furthermore, the generalized Eta constants, $\eta^{\left(n\right)}$,
should not be confused with the eta functions of Dirichlet, Dedekind,
or Weierstrass.
\end{singlespace}

\subsection{Asymptotic Behavior}

\begin{singlespace}
\noindent\phantom{}\medskip{}

\noindent Table 1 presents the first 16 values of the generalized
Euler-Mascheroni, Euler-Gompertz, and Eta constants. As is clear from
a brief inspection, both the $\gamma^{\left(n\right)}$ and $\eta^{\left(n\right)}$
increase approximately factorially, whereas the $\delta^{\left(n\right)}$,
which alternate in sign, grow super-exponentially but sub-factorially
in magnitude.
\end{singlespace}

\begin{singlespace}
The following proposition quantifies the asymptotic behavior underlying
these observations.
\end{singlespace}
\begin{prop}
\begin{singlespace}
As $n\rightarrow\infty$:
\end{singlespace}

\begin{singlespace}
\noindent (i) $\gamma^{\left(n\right)}=n!\left(1-1/2^{n+1}+O\left(1/3^{n}\right)\right)$;

\noindent (ii) $\left|\delta^{\left(n\right)}\right|=e\left[W\left(n\right)\right]^{n}\exp\left(-n/W\left(n\right)\right)\sqrt{2\pi n/\left(W\left(n\right)+1\right)}\left(1+o\left(1\right)\right)$,
where $W\left(n\right)$ denotes the principal branch of the Lambert
W function for $n>0$; and

\noindent (iii) $\eta^{\left(n\right)}=n!\left(1-1/2^{n+1}+O\left(1/3^{n}\right)\right)$.
\end{singlespace}
\end{prop}
\begin{proof}
\phantom{}\medskip{}

\begin{singlespace}
\noindent See the Appendix.
\end{singlespace}
\end{proof}
\begin{singlespace}
\begin{center}
Table 1. Values of $\gamma^{\left(n\right)}$, $\delta^{\left(n\right)}$,
and $\eta^{\left(n\right)}$ for $n\in\left\{ 0,1,\ldots,15\right\} $
\par\end{center}

\begin{center}
\begin{tabular}{|c|c|c|c|}
\hline 
{\footnotesize$\boldsymbol{n}$} & {\footnotesize$\boldsymbol{\gamma^{\left(n\right)}}$} & {\footnotesize$\boldsymbol{\delta^{\left(n\right)}}$} & {\footnotesize$\boldsymbol{\eta^{\left(n\right)}}$}\tabularnewline
\hline 
\hline 
{\footnotesize$0$} & {\footnotesize$1$} & {\footnotesize$-1$} & {\footnotesize$0.6321205588\ldots$}\tabularnewline
\hline 
{\footnotesize$1$} & {\footnotesize$0.5772156649\ldots$} & {\footnotesize$0.5963473623\ldots$} & {\footnotesize$0.7965995992\ldots$}\tabularnewline
\hline 
{\footnotesize$2$} & {\footnotesize$1.9781119906\ldots$} & {\footnotesize$-0.5319307700\ldots$} & {\footnotesize$1.7824255962\ldots$}\tabularnewline
\hline 
{\footnotesize$3$} & {\footnotesize$5.4448744564\ldots$} & {\footnotesize$0.5806819508\ldots$} & {\footnotesize$5.6584954080\ldots$}\tabularnewline
\hline 
{\footnotesize$4$} & {\footnotesize$23.5614740840\ldots$} & {\footnotesize$-0.7222515339\ldots$} & {\footnotesize$23.2957725933\ldots$}\tabularnewline
\hline 
{\footnotesize$5$} & {\footnotesize$117.8394082683\ldots$} & {\footnotesize$0.9875880596\ldots$} & {\footnotesize$118.2027216118\ldots$}\tabularnewline
\hline 
{\footnotesize$6$} & {\footnotesize$715.0673625273\ldots$} & {\footnotesize$-1.4535032853\ldots$} & {\footnotesize$714.5326485509\ldots$}\tabularnewline
\hline 
{\footnotesize$7$} & {\footnotesize$5,019.8488726298\ldots$} & {\footnotesize$2.2708839827\ldots$} & {\footnotesize$5,020.6842841603\ldots$}\tabularnewline
\hline 
{\footnotesize$8$} & {\footnotesize$40,243.6215733357\ldots$} & {\footnotesize$-3.7298791058\ldots$} & {\footnotesize$40,242.2494274946\ldots$}\tabularnewline
\hline 
{\footnotesize$9$} & {\footnotesize$362,526.2891146549\ldots$} & {\footnotesize$6.3945118625\ldots$} & {\footnotesize$362,528.6415241055\ldots$}\tabularnewline
\hline 
{\footnotesize$10$} & {\footnotesize$3,627,042.4127568947\ldots$} & {\footnotesize$-11.3803468877\ldots$} & {\footnotesize$3,627,038.2261612413\ldots$}\tabularnewline
\hline 
{\footnotesize$11$} & {\footnotesize$39,907,084.1514313358\ldots$} & {\footnotesize$20.9346984188\ldots$} & {\footnotesize$39,907,091.8528764912\ldots$}\tabularnewline
\hline 
{\footnotesize$12$} & {\footnotesize$478,943,291.7651829432\ldots$} & {\footnotesize$-39.6671864816\ldots$} & {\footnotesize$478,943,277.1724405475\ldots$}\tabularnewline
\hline 
{\footnotesize$13$} & {\footnotesize$6,226,641,351.5460642549\ldots$} & {\footnotesize$77.1984745660\ldots$} & {\footnotesize$6,226,641,379.9457959376\ldots$}\tabularnewline
\hline 
{\footnotesize$14$} & {\footnotesize$87,175,633,810.7084156319\ldots$} & {\footnotesize$-153.9437943882\ldots$} & {\footnotesize$87,175,633,754.0756585806\ldots$}\tabularnewline
\hline 
{\footnotesize$15$} & {\footnotesize$1,307,654,429,495.7941762096\ldots$} & {\footnotesize$313.9164765016\ldots$} & {\footnotesize$1,307,654,429,611.2775941595\ldots$}\tabularnewline
\hline 
\end{tabular}\medskip{}
\par\end{center}
\end{singlespace}

\subsection{Outline of Results}

\begin{singlespace}
\noindent\phantom{}\medskip{}

\noindent In Section 2, we begin our study of the arithmetic properties
of the three sequences of generalized constants by investigating $\gamma^{\left(n\right)}=\left(-1\right)^{n}\Gamma^{\left(n\right)}\left(1\right)$
for $n\in\mathbb{Z}_{\geq1}$ through generating functions of the
$\textrm{Gumbel}\left(0,1\right)$ distribution and various derivative
identities. This analysis reveals that $\gamma^{\left(2n\right)}\in\mathbb{Q}\left[\gamma,\gamma^{\left(2\right)},\gamma^{\left(3\right)},\ldots,\gamma^{\left(2n-1\right)}\right]$
for $n\in\mathbb{Z}_{\geq2}$ and that $\gamma^{\left(n\right)}$
is transcendental infinitely often with the density of transcendental
values among $n\in\left\{ 1,2,\ldots,N\right\} $ bounded below by
$\beta\left(N\right)=\max\left\{ 0,\sqrt{N-1}-3/2\right\} /N$, where
$\beta\left(N\right)\asymp N^{-1/2}$ as $N\rightarrow\infty$. In
Section 3, we employ a theorem of Shidlovskii (1989) to demonstrate
the algebraic independence, and therefore transcendence, of the $\eta^{\left(n\right)}$
for all $n\in\mathbb{Z}_{\geq0}$. This implies that at least one
element of each pair, $\left\{ \gamma^{\left(n\right)},\delta^{\left(n\right)}/e\right\} $
and $\left\{ \gamma^{\left(n\right)},\delta^{\left(n\right)}\right\} $,
and at least two elements of the triple $\left\{ \gamma^{\left(n\right)},\delta^{\left(n\right)}/e,\delta^{\left(n\right)}\right\} $
are transcendental for all $n\in\mathbb{Z}_{\geq1}$. Moreover, further
analysis of the $\gamma^{\left(n\right)}$ and $\eta^{\left(n\right)}$
shows that both the $\delta^{\left(n\right)}/e$ and $\delta^{\left(n\right)}$
are transcendental infinitely often, with lower asymptotic densities
of at least $1/2$. Finally, Section 4 addresses two additional sequences,
$\widetilde{\delta}^{\left(n\right)}$ and $\widetilde{\eta}^{\left(n\right)}$,
satisfying the ``non-alternating analogue'' equation $\gamma^{\left(n\right)}+\widetilde{\delta}^{\left(n\right)}/e=\widetilde{\eta}^{\left(n\right)}$,
and provides results parallel to those for the original system.
\end{singlespace}

\section{Transcendence Properties of the $\gamma^{\left(n\right)}$}

\subsection{Generating Functions and Identities}

\begin{singlespace}
\noindent\phantom{}\medskip{}

\noindent For $X\sim\textrm{Gumbel}\left(0,1\right)$, it is well
known that:
\end{singlespace}

(I) the moment-generating function, $\textrm{M}_{X}\left(t\right),\:\left|t\right|<1$,
is given by
\[
\textrm{E}_{X}\left[e^{tX}\right]=\Gamma\left(1-t\right)
\]
\begin{equation}
={\displaystyle \sum_{i=0}^{\infty}}\textrm{E}_{X}\left[X^{i}\right]\dfrac{t^{i}}{i!}=1+{\displaystyle \sum_{i=1}^{\infty}}\gamma^{\left(i\right)}\dfrac{t^{i}}{i!};
\end{equation}
and

(II) the associated cumulant-generating function, $\textrm{K}_{X}\left(t\right),\:\left|t\right|<1$,
is given by
\[
\ln\left(\textrm{E}_{X}\left[e^{tX}\right]\right)=\ln\left(\Gamma\left(1-t\right)\right)
\]
\begin{equation}
={\displaystyle \sum_{j=1}^{\infty}}\kappa_{j}\dfrac{t^{j}}{j!}=\gamma t+{\displaystyle \sum_{j=2}^{\infty}}\zeta\left(j\right)\dfrac{t^{j}}{j},
\end{equation}
where
\[
\kappa_{j}=\begin{cases}
\gamma, & j=1\\
\left(j-1\right)!\zeta\left(j\right), & j\in\mathbb{Z}_{\geq2}
\end{cases}
\]

\begin{singlespace}
\noindent denotes the $j^{\textrm{th}}$ cumulant.
\end{singlespace}

\begin{singlespace}
Setting
\begin{equation}
\gamma t+{\displaystyle \sum_{j=2}^{\infty}}\zeta\left(j\right)\dfrac{t^{j}}{j}=\ln\left(1+{\displaystyle \sum_{i=1}^{\infty}}\gamma^{\left(i\right)}\dfrac{t^{i}}{i!}\right)
\end{equation}
from (10) and (11), one can differentiate both sides of (12) with
respect to $t$ to obtain
\[
\gamma+{\displaystyle \sum_{j=2}^{\infty}}\zeta\left(j\right)t^{j-1}=\dfrac{{\displaystyle \sum_{i=1}^{\infty}}\gamma^{\left(i\right)}\dfrac{t^{i-1}}{\left(i-1\right)!}}{1+{\displaystyle \sum_{i=1}^{\infty}}\gamma^{\left(i\right)}\dfrac{t^{i}}{i!}}
\]
\[
\Longrightarrow\left(\gamma+{\displaystyle \sum_{j=2}^{\infty}}\zeta\left(j\right)t^{j-1}\right)\left(1+{\displaystyle \sum_{i=1}^{\infty}}\gamma^{\left(i\right)}\dfrac{t^{i}}{i!}\right)={\displaystyle \sum_{i=1}^{\infty}}\gamma^{\left(i\right)}\dfrac{t^{i-1}}{\left(i-1\right)!}
\]
\[
\Longrightarrow\gamma+\gamma{\displaystyle \sum_{i=1}^{\infty}}\gamma^{\left(i\right)}\dfrac{t^{i}}{i!}+{\displaystyle \sum_{j=2}^{\infty}}\zeta\left(j\right)t^{j-1}+\left({\displaystyle \sum_{j=2}^{\infty}}\zeta\left(j\right)t^{j-1}\right)\left({\displaystyle \sum_{i=1}^{\infty}}\gamma^{\left(i\right)}\dfrac{t^{i}}{i!}\right)={\displaystyle \sum_{i=1}^{\infty}}\gamma^{\left(i\right)}\dfrac{t^{i-1}}{\left(i-1\right)!}
\]
\begin{equation}
\Longrightarrow\gamma+\gamma{\displaystyle \sum_{m=1}^{\infty}}\gamma^{\left(m\right)}\dfrac{t^{m}}{m!}+{\displaystyle \sum_{m=1}^{\infty}}\zeta\left(m+1\right)t^{m}+{\displaystyle \sum_{m=2}^{\infty}}\left[{\displaystyle \sum_{i=1}^{m-1}}\dfrac{\zeta\left(m-i+1\right)\gamma^{\left(i\right)}}{i!}\right]t^{m}={\displaystyle \sum_{m=0}^{\infty}}\dfrac{\gamma^{\left(m+1\right)}}{m!}t^{m}.
\end{equation}
Then, matching coefficients of $t^{m}$ from both sides of (13) yields
\[
\begin{array}{c}
\gamma=\gamma^{\left(1\right)},\quad\textrm{for }m=0\\
\dfrac{\gamma\cdot\gamma^{\left(m\right)}}{m!}+\zeta\left(m+1\right)+{\displaystyle \sum_{i=1}^{m-1}}\dfrac{\zeta\left(m-i+1\right)\gamma^{\left(i\right)}}{i!}=\dfrac{\gamma^{\left(m+1\right)}}{m!},\quad\textrm{for }m\in\mathbb{Z}_{\geq1},
\end{array}
\]
which, by suitable changes of the index, can be rearranged to provide
the two distinct recurrence identities
\begin{equation}
\gamma^{\left(n\right)}=\gamma\cdot\gamma^{\left(n-1\right)}+{\displaystyle \sum_{i=0}^{n-2}}\dfrac{\left(n-1\right)!}{i!}\zeta\left(n-i\right)\gamma^{\left(i\right)},\quad n\in\mathbb{Z}_{\geq1}
\end{equation}
and
\begin{equation}
\left(n-1\right)!\zeta\left(n\right)=\gamma^{\left(n\right)}-\gamma\cdot\gamma^{\left(n-1\right)}-{\displaystyle \sum_{i=1}^{n-2}}\dfrac{\left(n-1\right)!}{\left(n-i-1\right)!}\zeta\left(i+1\right)\gamma^{\left(n-i-1\right)},\quad n\in\mathbb{Z}_{\geq2},
\end{equation}
respectively.

Next, we iterate (14) to find
\begin{equation}
\begin{array}{c}
\gamma^{\left(1\right)}=\gamma,\\
\gamma^{\left(2\right)}=\gamma^{2}+\zeta\left(2\right),\\
\gamma^{\left(3\right)}=\gamma^{3}+3\zeta\left(2\right)\gamma+2\zeta\left(3\right),\\
\gamma^{\left(4\right)}=\gamma^{4}+6\zeta\left(2\right)\gamma^{2}+8\zeta\left(3\right)\gamma+\dfrac{27}{2}\zeta\left(4\right),\\
\cdots,
\end{array}
\end{equation}
and solve (15) recursively to obtain\newpage\foreignlanguage{american}{
\begin{equation}
\begin{array}{c}
\zeta\left(2\right)=\gamma^{\left(2\right)}-\gamma^{2},\\
\zeta\left(3\right)=\dfrac{1}{2}\left(\gamma^{\left(3\right)}-3\gamma^{\left(2\right)}\gamma+2\gamma^{3}\right),\\
\zeta\left(4\right)=\dfrac{1}{6}\left[\gamma^{\left(4\right)}-4\gamma^{\left(3\right)}\gamma-3\left(\gamma^{\left(2\right)}\right)^{2}+12\gamma^{\left(2\right)}\gamma^{2}-6\gamma^{4}\right],\\
\zeta\left(5\right)=\dfrac{1}{24}\left[\gamma^{\left(5\right)}-5\gamma^{\left(4\right)}\gamma-10\gamma^{\left(3\right)}\gamma^{\left(2\right)}+20\gamma^{\left(3\right)}\gamma^{2}+30\left(\gamma^{\left(2\right)}\right)^{2}\gamma-60\gamma^{\left(2\right)}\gamma^{3}+24\gamma^{5}\right],\\
\cdots.
\end{array}
\end{equation}
}Then, substituting $6\zeta\left(2\right)$ for $\pi^{2}$ in the
even-zeta identity
\[
\zeta\left(2n\right)=\dfrac{\left(-1\right)^{n+1}\mathscr{B}_{2n}}{2\left(2n\right)!}\left(2\pi\right)^{2n}
\]
(where $\mathscr{B}_{2n}$ denotes the $\left(2n\right)^{\textrm{th}}$
Bernoulli number) gives
\begin{equation}
\zeta\left(2n\right)=\dfrac{\left(-1\right)^{n+1}2^{3n-1}3^{n}\mathscr{B}_{2n}}{\left(2n\right)!}\left(\zeta\left(2\right)\right)^{n}=\dfrac{\left(-1\right)^{n+1}2^{3n-1}3^{n}\mathscr{B}_{2n}}{\left(2n\right)!}\left(\gamma^{\left(2\right)}-\gamma^{2}\right)^{n},
\end{equation}
and one can set (18) equal to the expression for $\zeta\left(2n\right)$
generated by (17) for any $n\in\mathbb{Z}_{\geq2}$. After clearing
denominators, this yields a polynomial equation of the form,
\begin{equation}
P_{n}\left(\gamma,\gamma^{\left(2\right)},\gamma^{\left(3\right)},\ldots,\gamma^{\left(2n\right)}\right)=\alpha_{n,0}\gamma^{\left(2n\right)}-\alpha_{n,1}\gamma^{\left(2n-1\right)}\gamma-\sum_{i=0}^{2n}\left[\sum_{j=1}^{J_{i}}\alpha_{n,i,j}{\displaystyle \prod_{k=2}^{2n-2}}\left(\gamma^{\left(k\right)}\right)^{p_{n,k,i,j}}\right]\gamma^{i}=0
\end{equation}
for some $J_{i}\in\mathbb{Z}_{\geq0}$, $\alpha_{n,0},\alpha_{n,1}\in\mathbb{Z}_{\neq0}$,
$\alpha_{n,i,j}\in\mathbb{Z}$, and $p_{n,k,i,j}\in\mathbb{Z}_{\geq0}$.
In particular, for $n=2$ and $3$ we have\foreignlanguage{american}{
\[
P_{2}\left(\gamma,\gamma^{\left(2\right)},\gamma^{\left(3\right)},\gamma^{\left(4\right)}\right)=5\gamma^{\left(4\right)}-20\gamma^{\left(3\right)}\gamma-27\left(\gamma^{\left(2\right)}\right)^{2}+84\gamma^{\left(2\right)}\gamma^{2}-42\gamma^{4}=0
\]
and}
\end{singlespace}

\selectlanguage{american}%
\begin{singlespace}
\[
\begin{array}{c}
P_{3}\left(\gamma,\gamma^{\left(2\right)},\gamma^{\left(3\right)},\gamma^{\left(4\right)},\gamma^{\left(5\right)},\gamma^{\left(6\right)}\right)=7\gamma^{\left(6\right)}-42\gamma^{\left(5\right)}\gamma+210\gamma^{\left(4\right)}\gamma^{2}-105\gamma^{\left(4\right)}\gamma^{\left(2\right)}-70\left(\gamma^{\left(3\right)}\right)^{2}\\
+840\gamma^{\left(3\right)}\gamma^{\left(2\right)}\gamma-840\gamma^{\left(3\right)}\gamma^{3}+18\left(\gamma^{\left(2\right)}\right)^{3}-1314\left(\gamma^{\left(2\right)}\right)^{2}\gamma^{2}+1944\gamma^{\left(2\right)}\gamma^{4}-648\gamma^{6}=0,
\end{array}
\]
respectively.

\selectlanguage{english}%
By replacing all components of Euler's reflection formula by their
Taylor-series expansions, we further find
\[
\Gamma\left(1+t\right)\Gamma\left(1-t\right)=\dfrac{\pi t}{\sin\left(\pi t\right)}
\]
\[
\Longrightarrow\left({\displaystyle \sum_{\ell=0}^{\infty}}\gamma^{\left(\ell\right)}\dfrac{\left(-t\right)^{\ell}}{\ell!}\right)\left({\displaystyle \sum_{i=0}^{\infty}}\gamma^{\left(i\right)}\dfrac{t^{i}}{i!}\right)={\displaystyle \sum_{j=0}^{\infty}}c_{j}\pi^{2j}t^{2j}
\]
\begin{equation}
\Longrightarrow{\displaystyle \sum_{k=0}^{\infty}}\left[{\displaystyle \sum_{\ell=0}^{2k}}\left(-1\right)^{\ell}\dfrac{\gamma^{\left(\ell\right)}\gamma^{\left(2k-\ell\right)}}{\ell!\left(2k-\ell\right)!}\right]t^{2k}={\displaystyle \sum_{k=0}^{\infty}}c_{k}\pi^{2k}t^{2k},
\end{equation}
where $c_{k}\in\mathbb{Q}^{\times}$ for all $k\in\mathbb{Z}_{\geq0}$.
Then, matching coefficients of $t^{2k}$ from both sides of (20) implies
the identity
\begin{equation}
{\displaystyle \sum_{j=0}^{2k}}\dfrac{\left(-1\right)^{j}\gamma^{\left(j\right)}\gamma^{\left(2k-j\right)}}{j!\left(2k-j\right)!}=c_{k}\pi^{2k}.
\end{equation}

\end{singlespace}

\subsection{Transcendence Results}

\begin{singlespace}
\noindent\phantom{}\medskip{}

\noindent The following theorem, based on the analysis of Section
2.1, provides the main results of Section 2.
\end{singlespace}
\selectlanguage{american}%
\begin{thm}
For the sequence \foreignlanguage{english}{$\left\{ \gamma^{\left(n\right)}\right\} _{n\geq1}$:}

\selectlanguage{english}%
\begin{singlespace}
\noindent (i) at least one element of the pair $\left\{ \gamma,\gamma^{\left(2\right)}\right\} $
is transcendental;

\noindent (ii) $\gamma^{\left(2n\right)}\in\mathbb{Q}\left[\gamma,\gamma^{\left(2\right)},\gamma^{\left(3\right)},\ldots,\gamma^{\left(2n-1\right)}\right]$
for all $n\in\mathbb{Z}_{\geq2}$;

\noindent (iii) $\gamma^{\left(n\right)}$ is transcendental infinitely
often; and

\noindent (iv) $\#\left\{ n\in\left\{ 1,2,\ldots,N\right\} :\gamma^{\left(n\right)}\textrm{ is transcendental}\right\} /N\geq\beta\left(N\right)=\max\left\{ 0,\sqrt{N-1}-3/2\right\} /N$
for all $N\in\mathbb{Z}_{\geq1}$, where $\beta\left(N\right)\asymp N^{-1/2}$
as $N\rightarrow\infty$.
\end{singlespace}
\end{thm}
\selectlanguage{english}%
\begin{proof}
\phantom{}\medskip{}

\begin{singlespace}
\noindent (i) From the second line of (16), we see that $\gamma^{\left(2\right)}=\gamma^{2}+\zeta\left(2\right)$.
Since $\zeta\left(2\right)=\pi^{2}/6$ is transcendental, it follows
immediately that if either $\gamma$ or $\gamma^{\left(2\right)}$
is algebraic, then the other member of the pair must be transcendental.

\noindent (ii) Rewriting the polynomial equation of (19), we find
\[
\gamma^{\left(2n\right)}=\dfrac{\alpha_{n,1}}{\alpha_{n,0}}\gamma^{\left(2n-1\right)}\gamma+\sum_{i=0}^{2n}\left[\sum_{j=1}^{J_{i}}\dfrac{\alpha_{n,i,j}}{\alpha_{n,0}}{\displaystyle \prod_{k=2}^{2n-2}}\left(\gamma^{\left(k\right)}\right)^{p_{n,k,i,j}}\right]\gamma^{i},
\]
where $J_{i}\in\mathbb{Z}_{\geq0}$, $\alpha_{n,0},\alpha_{n,1}\in\mathbb{Z}_{\neq0}$,
$\alpha_{n,i,j}\in\mathbb{Z}$, and $p_{n,k,i,j}\in\mathbb{Z}_{\geq0}$.
Thus, $\gamma^{\left(2n\right)}$ can be expressed as a polynomial
with rational coefficients in the variables $\gamma,\gamma^{\left(2\right)},\gamma^{\left(3\right)},\ldots,\gamma^{\left(2n-1\right)}$
for all $n\in\mathbb{Z}_{\geq2}$.
\end{singlespace}

\begin{singlespace}
\noindent (iii) Suppose (for purposes of contradiction) that only
finitely many $\gamma^{\left(n\right)}$ are transcendental, and choose
$n^{*}\in\mathbb{Z}_{\geq1}$ such that $\gamma^{\left(n\right)}\in\overline{\mathbb{Q}}$
for all $n>n^{*}$, with $\mathcal{W}=\textrm{span}_{\overline{\mathbb{Q}}}\left\{ \gamma^{\left(0\right)},\gamma,\gamma^{\left(2\right)},\gamma^{\left(3\right)},\ldots,\gamma^{\left(n^{*}\right)}\right\} $
denoting a finite-dimensional $\overline{\mathbb{Q}}$-vector space.
Now note that for any $k>n^{*}$, each summand on the left-hand side
of (21) has at least one algebraic factor $\gamma^{\left(n_{k}\right)}$,
for some $n_{k}>n^{*}$ (because $j+\left(2k-j\right)=2k>2n^{*}\Longrightarrow\left(j>n^{*}\right)\vee\left(2k-j>n^{*}\right)$).
It then follows that every summand lies in $\mathcal{W}$, and therefore
that $c_{k}\pi^{2k}\in\mathcal{W}$ for all $k>n^{*}$. Since $c_{k}\in\mathbb{Q}^{\times}$,
we see that $\pi^{2k}\in\mathcal{W}$ for all $k>n^{*}$, and thus
that all $n^{*}+2$ elements $\left\{ \pi^{2\left(n^{*}+1\right)},\pi^{2\left(n^{*}+2\right)},\ldots,\pi^{2\left(2n^{*}+2\right)}\right\} $
must lie in the space $\mathcal{W}$. However, the dimension of $\mathcal{W}$
is at most $n^{*}+1$, forcing a non-trivial $\overline{\mathbb{Q}}$-linear
dependence among these $n^{*}+2$ elements; that is, a non-trivial
polynomial relation over $\overline{\mathbb{Q}}$ satisfied by $\pi^{2}$,
which contradicts the transcendence of $\pi$.

\noindent (iv) Fixing $N\in\mathbb{Z}_{\geq1}$, consider the set
of terms $\left\{ c_{k}\pi^{2k}\right\} $ for $k\in\left\{ 0,1,\ldots,\left\lfloor N/2\right\rfloor \right\} $
given by the right-hand side of (21). From the left-hand side of the
same equation, we know that each $c_{k}\pi^{2k}$ is a $\overline{\mathbb{Q}}$-linear
combination of products $\gamma^{\left(u\right)}\gamma^{\left(v\right)}$,
with $0\leq u,v\le N$. Now assume that the number of transcendental
elements $\gamma^{\left(n\right)}$, for $n\in\left\{ 1,2,\ldots,N\right\} $,
is given by $\tau$ for some integer $\tau\geq0$, and let $\theta_{1},\theta_{2},\dots,\theta_{\tau^{*}}$
denote the $\tau^{*}\leq\tau$ distinct transcendental numbers within
the set $\left\{ \gamma,\gamma^{\left(2\right)},\gamma^{\left(3\right)},\ldots,\gamma^{\left(N\right)}\right\} $,
so that every element of this set lies in $\overline{\mathbb{Q}}\cup\left\{ \theta_{1},\theta_{2},\dots,\theta_{\tau^{*}}\right\} $.
It then follows that every product $\gamma^{\left(u\right)}\gamma^{\left(v\right)}$
(with $0\leq u,v\le N$) is contained in the $\overline{\mathbb{Q}}$-vector
space\foreignlanguage{american}{
\[
\mathcal{U}_{\tau^{*}}=\textrm{span}_{\overline{\mathbb{Q}}}\left\{ 1,\theta_{1},\theta_{2},\dots,\theta_{\tau^{*}},\theta_{i}\theta_{i^{\prime}}:1\leq i\leq i^{\prime}\leq\tau^{*}\right\} ,
\]
implying} $c_{k}\pi^{2k}\in\mathcal{U}_{\tau^{*}}$.
\end{singlespace}

Since the generating set for $\mathcal{U}_{\tau^{*}}$ has cardinality
\[
1+\tau^{*}+\frac{\tau^{*}\left(\tau^{*}+1\right)}{2}=\frac{\left(\tau^{*}+1\right)\left(\tau^{*}+2\right)}{2}\leq\frac{\left(\tau+1\right)\left(\tau+2\right)}{2},
\]
one can see that 
\[
\dim_{\overline{\mathbb{Q}}}\left(\mathcal{U}_{\tau^{*}}\right)\le\frac{\left(\tau+1\right)\left(\tau+2\right)}{2}.
\]
However, we also know from above that the set $\left\{ c_{0},c_{1}\pi^{2},\dots,c_{\left\lfloor N/2\right\rfloor }\pi^{2\left\lfloor N/2\right\rfloor }\right\} $
are $\overline{\mathbb{Q}}$-linearly independent, forcing the span
of these elements to have dimension $\left\lfloor N/2\right\rfloor +1$.
Given that this span is contained in $\mathcal{U}_{\tau^{*}}$, it
follows that
\[
\left\lfloor N/2\right\rfloor +1\le\dim_{\overline{\mathbb{Q}}}\left(\mathcal{U}_{\tau^{*}}\right)\le\frac{\left(\tau+1\right)\left(\tau+2\right)}{2}
\]
\[
\Longrightarrow2\left\lfloor N/2\right\rfloor +2\le\left(\tau+1\right)\left(\tau+2\right)
\]
\[
\Longrightarrow\tau\ge\sqrt{2\left\lfloor N/2\right\rfloor +9/4}-3/2.
\]
Thus,
\[
\tau\ge\left\lceil \sqrt{2\left\lfloor N/2\right\rfloor +9/4}-3/2\right\rceil \geq\max\left\{ 0,\sqrt{N-1}-3/2\right\} 
\]
\[
\Longrightarrow\dfrac{\tau}{N}\geq\beta\left(N\right)=\dfrac{\max\left\{ 0,\sqrt{N-1}-3/2\right\} }{N},
\]
where $\beta\left(N\right)\asymp N^{-1/2}$ as $N\rightarrow\infty$.
\end{proof}
\begin{singlespace}
\newpage The identity $\gamma^{\left(n\right)}=\left(-1\right)^{n}\Gamma^{\left(n\right)}\left(1\right)$
implies that all four parts of Theorem 1 are equally valid if phrased
in terms of Gamma-function derivatives rather than generalized Euler-Mascheroni
constants. In particular, $\Gamma^{\left(n\right)}\left(1\right)$
is transcendental infinitely often, a result consistent with the common
expectation that $\Gamma^{\left(n\right)}\left(1\right)$ is transcendental
for all $n\geq1$ (see, e.g., Rivoal {[}8{]} and Fischler and Rivoal
{[}9{]}).\footnote{It may be of interest to note that $\beta\left(N\right)$ is larger
in magnitude than the asymptotic lower bound for the density of irrational
$\zeta\left(n\right)$ among $n\in\left\{ 3,5,\ldots,N\right\} $
given by Fischler, Sprang, and Zudilin {[}10{]}:
\[
\beta_{\zeta}\left(N\right)\asymp\dfrac{2^{\left(1-\varepsilon\right)\ln\left(N\right)/\ln\left(\ln\left(N\right)\right)}}{\left(N-1\right)/2}\asymp N^{\left(1-\varepsilon\right)\ln\left(2\right)/\ln\left(\ln\left(N\right)\right)-1},
\]
for arbitrarily small $\varepsilon>0$.}
\end{singlespace}

\section{Algebraic Independence of the $\eta^{\left(n\right)}$}

\noindent Although the generalized Eta constants do not satisfy convenient
identities such as those derived for the generalized Euler-Mascheroni
constants in Section 2.1, they actually admit more powerful methods
for the study of their arithmetic properties. Specifically, we can
embed $\eta^{\left(n\right)}=-n!{\textstyle \sum_{k=1}^{\infty}}{\displaystyle \left(-1\right)^{k}/\left(k^{n}\cdot k!\right)}$
into the sequence of functions
\begin{equation}
F_{n}\left(z\right)=-n!{\displaystyle \sum_{k=1}^{\infty}}\dfrac{z^{k}}{k^{n}\cdot k!},\quad n\in\mathbb{Z}_{\geq0},z\in\mathbb{C},
\end{equation}
which are essentially constant multiples of the polylogarithm-exponential
series first studied by Hardy {[}11{]} and later used extensively
by Shidlovskii {[}12{]} in his analysis of $E$-functions. For the
purposes at hand, we note that these functions are readily matched
to the system $w_{n}\left(z\right)=1+{\textstyle \sum_{k=1}^{\infty}}z^{k}/\left(k^{n}\cdot k!\right)$
of Shidlovskii {[}12{]} (Chapter 7, Section 1) via the simple linear
identity
\[
F_{n}\left(z\right)\equiv n!\left(1-w_{n}\left(z\right)\right).
\]

This allows us to make two assertions:

\begin{singlespace}
(I) The $F_{n}\left(z\right)$, like the $w_{n}\left(z\right)$, are
$E$-functions (in the sense of Siegel), each of which satisfies an
$\left(n+1\right)^{\textrm{st}}$-order non-homogeneous linear differential
equation with coefficients in $\mathbb{Q}\left(z\right)$ and a unique
finite singularity at $z=0$ (see Chapter 7, Section 1 of Shidlovskii
{[}12{]}).

(II) For any $m\in\mathbb{Z}_{\geq1}$ and $z\in\overline{\mathbb{Q}}^{\times}$,
the set of values $\left\{ F_{0}\left(z\right),F_{1}\left(z\right),\ldots,F_{m}\left(z\right)\right\} $,
like the set of values $\left\{ w_{0}\left(z\right),w_{1}\left(z\right),\ldots,w_{m}\left(z\right)\right\} $,
are algebraically independent over $\overline{\mathbb{Q}}$ (see Chapter
7, Section 1, Theorem 1 of Shidlovskii {[}12{]}; the theorem is restated
and proved in Chapter 8, Section 3 of the same volume).\footnote{These assertions do not explicitly mention that the functions $F_{0}\left(z\right),F_{1}\left(z\right),\ldots,F_{m}\left(z\right)$
are themselves algebraically independent over $\overline{\mathbb{Q}}\left(z\right)$
-- a common step in the conventional Siegel-Shidlovskii framework
(see, e.g., Beukers {[}13{]}). However, the present formulation is
faithful to the indicated theorem of Shidlovskii {[}12{]}, which is
tailored for the specific family of $E$-functions considered.}

As an immediate consequence of assertions (I) and (II), we obtain
the following theorem.
\end{singlespace}
\begin{thm}
The set of values $\left\{ \eta^{\left(n\right)}\right\} _{n\geq0}$
are algebraically independent over $\overline{\mathbb{Q}}$, implying
that $\eta^{\left(n\right)}$ is transcendental for all $n\in\mathbb{Z}_{\geq0}$.
\end{thm}
\begin{proof}
\phantom{} \medskip{}

\begin{singlespace}
\noindent This result follows directly from the above-mentioned Theorem
1 of Shidlovskii {[}12{]}. First, set $z=-1$ in assertion (II) above
to show that, for any $m\in\mathbb{Z}_{\geq1}$, the set of values
$\left\{ \eta^{\left(0\right)},\eta^{\left(1\right)},\ldots,\eta^{\left(m\right)}\right\} $
are algebraically independent over $\overline{\mathbb{Q}}$. Next,
note that the infinite set $\left\{ \eta^{\left(n\right)}\right\} _{n\geq0}$
are algebraically independent if and only if the subsets $\left\{ \eta^{\left(0\right)},\eta^{\left(1\right)},\ldots,\eta^{\left(m\right)}\right\} $
are algebraically independent for all finite $m$.
\end{singlespace}
\end{proof}
\begin{singlespace}
As indicated above, Theorem 2 is a direct consequence of Shidlovskii's
theorem. However, we have been unable to find an application of this
theorem to the generalized positive partial $\textrm{Gumbel}\left(0,1\right)$
moments of (7) in the literature. The following corollary shows the
impact of the transcendence of the generalized Eta constants on the
corresponding generalized Euler-Mascheroni and Euler-Gompertz constants.
\end{singlespace}
\begin{cor}
For all $n\in\mathbb{Z}_{\geq1}$:

\begin{singlespace}
\noindent (i) at least one element of the pair $\left\{ \gamma^{\left(n\right)},\delta^{\left(n\right)}/e\right\} $
is transcendental;

\noindent (ii) at least one element of the pair $\left\{ \gamma^{\left(n\right)},\delta^{\left(n\right)}\right\} $
is transcendental; and

\noindent (iii) at least two elements of the triple $\left\{ \gamma^{\left(n\right)},\delta^{\left(n\right)}/e,\delta^{\left(n\right)}\right\} $
are transcendental.
\end{singlespace}
\end{cor}
\begin{proof}
\phantom{} \medskip{}

\begin{singlespace}
\noindent (i) Theorem 2 states that $\eta^{\left(n\right)}$ is transcendental
for all $n\in\mathbb{Z}_{\geq0}$. Therefore, it is clear from (9)
that the two numbers $\gamma^{\left(n\right)}$ and $\delta^{\left(n\right)}/e$
cannot both be algebraic. (We omit the case of $n=0$ from the statement
of part (i) because it is obvious that $\delta^{\left(0\right)}/e=-1/e$
is transcendental.)

\noindent (ii) From Theorem 2, we know that $\eta^{\left(n\right)}=\gamma^{\left(n\right)}+\delta^{\left(n\right)}/e$
is algebraically independent of $\eta^{\left(0\right)}=1-1/e$ for
all $n\in\mathbb{Z}_{\geq1}$. Therefore, if we assume (for purposes
of contradiction) that both $\gamma^{\left(n\right)}$ and $\delta^{\left(n\right)}$
are algebraic, then we can write 
\[
\eta^{\left(n\right)}=\gamma^{\left(n\right)}+\delta^{\left(n\right)}\left(1-\eta^{\left(0\right)}\right)
\]
\[
\Longleftrightarrow\eta^{\left(n\right)}-\left(\gamma^{\left(n\right)}+\delta^{\left(n\right)}\right)+\delta^{\left(n\right)}\eta^{\left(0\right)}=0
\]
\[
\Longleftrightarrow a\eta^{\left(n\right)}+b\eta^{\left(0\right)}+c=0,
\]
where $a,b,c\in\overline{\mathbb{Q}}$. This contradicts the algebraic
independence of $\eta^{\left(n\right)}$ and $\eta^{\left(0\right)}$,
forcing at least one of $\left\{ \gamma^{\left(n\right)},\delta^{\left(n\right)}\right\} $
to be transcendental. (We note that this result is implied by Section
7 of Rivoal {[}4{]}, where the author expands on earlier work of Mahler
{[}3{]}. Rivoal addresses the special case of $n=1$ -- mentioned
in Section 1.1 of the present article -- in Section 1 of {[}4{]}.)

\noindent (iii) Consider the pair $\left\{ \delta^{\left(n\right)}/e,\delta^{\left(n\right)}\right\} $.
Since $e$ is transcendental and $\delta^{\left(n\right)}\neq0$,
one can see that if either $\delta^{\left(n\right)}/e$ or $\delta^{\left(n\right)}$
is algebraic, then the other member of the pair must be transcendental.
Combining this disjunctive transcendence of $\left\{ \delta^{\left(n\right)}/e,\delta^{\left(n\right)}\right\} $
with the disjunctive transcendence of both pairs $\left\{ \gamma^{\left(n\right)},\delta^{\left(n\right)}/e\right\} $
and $\left\{ \gamma^{\left(n\right)},\delta^{\left(n\right)}\right\} $
(from parts (i) and (ii), respectively), it is easy to see that no
two elements of the triple $\left\{ \gamma^{\left(n\right)},\delta^{\left(n\right)}/e,\delta^{\left(n\right)}\right\} $
can both be algebraic.
\end{singlespace}
\end{proof}
\begin{singlespace}
For $n=1$, the disjunctive transcendence of $\left\{ \gamma,\delta/e\right\} $
complements the disjunctive transcendence of $\left\{ \gamma,\delta\right\} $
proved by Rivoal {[}4{]}. This result also draws attention to the
fact that, under the probabilistic interpretation of Hardy's equation,
the constant $\delta/e=\textrm{E}_{X}\left[X^{-}\right]=\textrm{E}_{X}\left[\max\left\{ -X,0\right\} \right]$
(for $X\sim\textrm{Gumbel}\left(0,1\right)$) is a more natural companion
of $\gamma=\textrm{E}_{X}\left[X\right]$ than is $\delta=-\textrm{E}_{X}\left[X\mid X\leq0\right]$.
In Powers {[}14{]}, we investigated the analytic relationship between
$\gamma$ and $\delta$ in (2) by considering linear combinations
of the form $\gamma+\alpha\delta$ for $\alpha\in\mathbb{R}^{\times}$.
Letting
\[
S_{\gamma}\coloneqq{\displaystyle \sum_{k=1}^{\infty}\dfrac{\left(-1\right)^{k}\left(!k\right)}{k}}
\]
and
\[
S_{\delta}\coloneqq{\displaystyle \sum_{k=1}^{\infty}\left(-1\right)^{k-1}\left(k-1\right)!}
\]
denote canonical Borel-summable divergent series for $\gamma$ and
$\delta$,\footnote{The divergent series $S_{\delta}$ is well-known for having been included
in the first letter of Srinivasa Ramanujan to G. H. Hardy (see Berndt
and Rankin {[}15{]}).} respectively (where $!k=k!{\textstyle \sum_{\ell=0}^{k}}\left(-1\right)^{\ell}/\ell!$
denotes the $k^{\textrm{th}}$ derangement number), it was found that
$\alpha=1/e$ is the unique coefficient such that the series $S_{\gamma}+\alpha S_{\delta}$
converges conventionally. This is because the Borel-transform kernels
of both $S_{\gamma}$ and $S_{\delta}$ are characterized by unique
logarithmic singularities at $-1$ with associated Stokes constants
(for the logarithmic-coefficient normalization at the singularity)
of $-1/e$ and $1$, respectively, thus forcing the divergent terms
to cancel when combined as $S_{\gamma}+S_{\delta}/e$.
\end{singlespace}

The following theorem relies on both our earlier analysis of the algebraic
structure of the polynomials $P_{n}$ (see Section 2.1) and the algebraic-independence
result of Theorem 2.
\selectlanguage{american}%
\begin{thm}
For the sequences \foreignlanguage{english}{$\left\{ \delta^{\left(n\right)}/e\right\} _{n\geq1}$
and $\left\{ \delta^{\left(n\right)}\right\} _{n\geq1}$:}

\selectlanguage{english}%
\begin{singlespace}
\noindent (i) $\delta^{\left(n\right)}/e$ is transcendental infinitely
often;

\noindent (ii) $\underset{N\rightarrow\infty}{\liminf}\left(\#\left\{ n\in\left\{ 1,2,\ldots,N\right\} :\delta^{\left(n\right)}/e\textrm{ is transcendental}\right\} /N\right)\geq1/2$;

\noindent (iii) $\delta^{\left(n\right)}$ is transcendental infinitely
often; and

\noindent (iv) $\underset{N\rightarrow\infty}{\liminf}\left(\#\left\{ n\in\left\{ 1,2,\ldots,N\right\} :\delta^{\left(n\right)}\textrm{ is transcendental}\right\} /N\right)\geq1/2$.\newpage{}
\end{singlespace}
\end{thm}
\selectlanguage{english}%
\begin{proof}
\phantom{}\medskip{}

\begin{singlespace}
\noindent (i) Define $\mathcal{G}_{n}=\left\{ \gamma,\gamma^{\left(2\right)},\gamma^{\left(3\right)},\ldots,\gamma^{\left(n\right)}\right\} $,
$\mathcal{H}_{n}=\left\{ \eta,\eta^{\left(2\right)},\eta^{\left(3\right)},\ldots,\eta^{\left(n\right)}\right\} $,
and $\mathcal{D}_{n}=\left\{ \delta/e,\delta^{\left(2\right)}/e,\delta^{\left(3\right)}/e,\ldots,\delta^{\left(n\right)}/e\right\} $
for $n\in\mathbb{Z}_{\geq1}$, and let $\mathcal{F}_{n}^{\left(\mathcal{G}\right)}=\mathbb{Q}\left(\mathcal{G}_{n}\right)$,
$\mathcal{F}_{n}^{\left(\mathcal{H}\right)}=\mathbb{Q}\left(\mathcal{H}_{n}\right)$,
and $\mathcal{F}_{n}^{\left(\mathcal{D}\right)}=\mathbb{Q}\left(\mathcal{D}_{n}\right)$
denote the respective fields generated by these three sets of numbers
over $\mathbb{Q}$. We wish to demonstrate that
\begin{equation}
\mathrm{trdeg}\left(\mathcal{F}_{n}^{\left(\mathcal{G}\right)}\right)\leq n-\lfloor n/2\rfloor+1
\end{equation}
for all $n\in\mathbb{Z}_{\geq1}$, and begin by noting that this inequality
holds for $n<4$ because $\mathcal{F}_{n}^{\left(\mathcal{G}\right)}$
is generated by at most $|\mathcal{G}_{n}|$ elements, and $|\mathcal{G}_{n}|=n\le n-\lfloor n/2\rfloor+1$
for $n\in\left\{ 1,2,3\right\} $. In the following analysis, we treat
the case of $n\geq4$.
\end{singlespace}

From (19), one can see that for any $m\in\left\{ 2,3,\ldots,n\right\} $,
$P_{m}$ is linear in the highest-index constant, $\gamma^{\left(2m\right)}$,
with a non-zero integer coefficient $\alpha_{m,0}$. This implies
that, as stated in Theorem 1(ii), each $\gamma^{\left(2m\right)}$
can be expressed as a polynomial with rational coefficients in strictly
lower-index constants; that is, $\gamma^{\left(2m\right)}\in\mathbb{Q}\left[\gamma,\gamma^{\left(2\right)},\gamma^{\left(3\right)},\ldots,\gamma^{\left(2m-1\right)}\right]$.
We now let
\[
\mathcal{B}_{n}=\mathcal{G}_{n}\setminus\left\{ \gamma^{\left(2m\right)}:2\leq m\leq\lfloor n/2\rfloor\right\} 
\]
and show by induction on $m$ that $\gamma^{\left(2m\right)}\in\mathbb{Q}\left(\mathcal{B}_{n}\right)$
for all $m\in\left\{ 2,3,\ldots,\lfloor n/2\rfloor\right\} $.

For $m=2$, the relation $P_{2}=0$ expresses $\gamma^{\left(4\right)}$
as a polynomial with rational coefficients in $\gamma,\gamma^{\left(2\right)},\gamma^{\left(3\right)}\in\mathcal{B}_{n}$,
implying $\gamma^{\left(4\right)}\in\mathbb{Q}\left(\mathcal{B}_{n}\right)$.
Now consider $m\ge3$, and assume that $\gamma^{\left(4\right)},\gamma^{\left(6\right)},\ldots,\gamma^{\left(2m-2\right)}\in\mathbb{Q}\left(\mathcal{B}_{n}\right)$.
Since (a) $P_{m}=0$ expresses $\gamma^{\left(2m\right)}$ as a polynomial
with rational coefficients in $\gamma,\gamma^{\left(2\right)},\gamma^{\left(3\right)},\ldots,\gamma^{\left(2m-1\right)}$,
(b) every odd-indexed element of the set $\left\{ \gamma^{\left(3\right)},\gamma^{\left(5\right)},\ldots,\gamma^{\left(2m-1\right)}\right\} $
lies in $\mathcal{B}_{n}$, and (c) every even-indexed element of
$\left\{ \gamma^{\left(4\right)},\gamma^{\left(6\right)},\ldots,\gamma^{\left(2m-2\right)}\right\} $
lies in $\mathbb{Q}\left(\mathcal{B}_{n}\right)$, it follows that
$\gamma^{\left(2m\right)}\in\mathbb{Q}\left(\mathcal{B}_{n}\right)$
as well. This confirms that $\gamma^{\left(2m\right)}\in\mathbb{Q}\left(\mathcal{B}_{n}\right)$
for all $m\in\left\{ 2,3,\ldots,\lfloor n/2\rfloor\right\} $, and
thus that $\mathcal{F}_{n}^{\left(\mathcal{G}\right)}=\mathbb{Q}\left(\mathcal{G}_{n}\right)=\mathbb{Q}\left(\mathcal{B}_{n}\right)$.

Given that $\mathcal{B}_{n}$ contains $n-\lfloor n/2\rfloor+1$ elements,
the transcendence degree of $\mathcal{F}_{n}^{\left(\mathcal{G}\right)}$
(which is generated by $\mathcal{B}_{n}$) satisfies (23). Then, from
the relation in (9), one can see that 
\[
\mathcal{F}_{n}^{\left(\mathcal{H}\right)}\subseteq\mathcal{F}_{n}^{\left(\mathcal{G}\right)}\mathcal{F}_{n}^{\left(\mathcal{D}\right)},
\]
implying 
\begin{equation}
\mathrm{trdeg}\left(\mathcal{F}_{n}^{\left(\mathcal{H}\right)}\right)\leq\mathrm{trdeg}\left(\mathcal{F}_{n}^{\left(\mathcal{G}\right)}\mathcal{F}_{n}^{\left(\mathcal{D}\right)}\right)\leq\mathrm{trdeg}\left(\mathcal{F}_{n}^{\left(\mathcal{G}\right)}\right)+\mathrm{trdeg}\left(\mathcal{F}_{n}^{\left(\mathcal{D}\right)}\right).
\end{equation}
Since the values in $\mathcal{H}_{n}$ are algebraically independent
by Theorem 2, we know that $\mathrm{trdeg}\left(\mathcal{F}_{n}^{\left(\mathcal{H}\right)}\right)=n$,
which (in conjunction with (23) and (24)) allows us to write
\[
n\leq n-\lfloor n/2\rfloor+1+\mathrm{trdeg}\left(\mathcal{F}_{n}^{\left(\mathcal{D}\right)}\right)
\]
\[
\Longleftrightarrow\mathrm{trdeg}\left(\mathcal{F}_{n}^{\left(\mathcal{D}\right)}\right)\geq\lfloor n/2\rfloor-1.
\]
Finally, since the transcendence degree of the field generated by
a set of numbers provides a lower bound on the number of transcendental
elements in the indicated set, and $\underset{n\rightarrow\infty}{\lim}\left(\lfloor n/2\rfloor-1\right)=\infty$,
it follows that the sequence $\left\{ \delta^{\left(n\right)}/e\right\} _{n\geq1}$
must contain an infinite number of transcendental terms.

\noindent (ii) The density of transcendental $\delta^{\left(n\right)}/e$
among $n\in\left\{ 1,2,\ldots,N\right\} $ satisfies
\[
\dfrac{\#\left\{ n\in\left\{ 1,2,\ldots,N\right\} :\delta^{\left(n\right)}/e\textrm{ is transcendental}\right\} }{N}\geq\frac{\mathrm{trdeg}\left(\mathcal{F}_{N}^{\left(\mathcal{D}\right)}\right)}{N}\geq\dfrac{\lfloor N/2\rfloor-1}{N}.
\]

\noindent Thus,
\[
\underset{N\rightarrow\infty}{\liminf}\left(\dfrac{\#\left\{ n\in\left\{ 1,2,\ldots,N\right\} :\delta^{\left(n\right)}/e\textrm{ is transcendental}\right\} }{N}\right)\geq\underset{N\rightarrow\infty}{\liminf}\left(\dfrac{\lfloor N/2\rfloor-1}{N}\right)\geq\dfrac{1}{2}.
\]

\begin{singlespace}
\noindent (iii) Retain $\mathcal{G}_{n}$ and $\mathcal{F}_{n}^{\left(\mathcal{G}\right)}$
as defined in the proof of part (i), and recall from Theorem 2 that
$\eta^{\left(n\right)}=\gamma^{\left(n\right)}+\delta^{\left(n\right)}/e$
is algebraically independent of $\eta^{\left(0\right)}=1-1/e$ for
all $n\in\mathbb{Z}_{\geq1}$. Writing
\begin{equation}
\eta^{\left(n\right)}=\gamma^{\left(n\right)}+\delta^{\left(n\right)}\left(1-\eta^{\left(0\right)}\right)
\end{equation}
\newpage{}

\noindent as in the proof of Corollary 1(ii), further define $\mathcal{H}_{n}^{\prime}=\left\{ \eta^{\left(0\right)},\eta,\eta^{\left(2\right)},\eta^{\left(3\right)},\ldots,\eta^{\left(n\right)}\right\} $
and\linebreak{}
$\mathcal{D}_{n}^{\prime}=\left\{ \delta,\delta^{\left(2\right)},\delta^{\left(3\right)},\ldots,\delta^{\left(n\right)}\right\} $
for $n\in\mathbb{Z}_{\geq1}$, where $\mathcal{F}_{n}^{\left(\mathcal{H}^{\prime}\right)}=\mathbb{Q}\left(\mathcal{H}_{n}^{\prime}\right)$
and $\mathcal{F}_{n}^{\left(\mathcal{D}^{\prime}\right)}=\mathbb{Q}\left(\mathcal{D}_{n}^{\prime}\right)$
denote the corresponding fields over $\mathbb{Q}$. From (25), we
know that
\[
\mathcal{F}_{n}^{\left(\mathcal{H}^{\prime}\right)}\subseteq\mathbb{Q}\left(\eta^{\left(0\right)},\mathcal{G}_{n}\right)\mathcal{F}_{n}^{\left(\mathcal{D}^{\prime}\right)},
\]
which implies 
\[
\mathrm{trdeg}\left(\mathcal{F}_{n}^{\left(\mathcal{H}^{\prime}\right)}\right)\leq\mathrm{trdeg}\left(\mathbb{Q}\left(\eta^{\left(0\right)},\mathcal{G}_{n}\right)\mathcal{F}_{n}^{\left(\mathcal{D}^{\prime}\right)}\right)\leq\mathrm{trdeg}\left(\mathbb{Q}\left(\eta^{\left(0\right)},\mathcal{G}_{n}\right)\right)+\mathrm{trdeg}\left(\mathcal{F}_{n}^{\left(\mathcal{D}^{\prime}\right)}\right).
\]
Then, since adjoining a single element to a field can increase the
transcendence degree by at most one, we have 
\[
\mathrm{trdeg}\left(\mathbb{Q}\left(\eta^{\left(0\right)},\mathcal{G}_{n}\right)\right)=\mathrm{trdeg}\left(\mathcal{F}_{n}^{\left(\mathcal{G}\right)}\left(\eta^{\left(0\right)}\right)\right)\leq\mathrm{trdeg}\left(\mathcal{F}_{n}^{\left(\mathcal{G}\right)}\right)+1\leq n-\lfloor n/2\rfloor+2,
\]
where the final inequality follows from (23). Moreover, we know from
Theorem 2 that the values in $\mathcal{H}_{n}^{\prime}$ are algebraically
independent, implying $\mathrm{trdeg}\left(\mathcal{F}_{n}^{\left(\mathcal{H}^{\prime}\right)}\right)=|\mathcal{H}_{n}^{\prime}|=n+1$
and therefore
\[
n+1\leq n-\lfloor n/2\rfloor+2+\mathrm{trdeg}\left(\mathcal{F}_{n}^{\left(\mathcal{D}^{\prime}\right)}\right)
\]
\[
\Longleftrightarrow\mathrm{trdeg}\left(\mathcal{F}_{n}^{\left(\mathcal{D}^{\prime}\right)}\right)\geq\lfloor n/2\rfloor-1.
\]
Since $\underset{n\rightarrow\infty}{\lim}\left(\lfloor n/2\rfloor-1\right)=\infty$,
it follows that the sequence $\left\{ \delta^{\left(n\right)}\right\} _{n\geq1}$
must contain an infinite number of transcendental terms.
\end{singlespace}

\begin{singlespace}
\noindent (iv) The density of transcendental $\delta^{\left(n\right)}$
among $n\in\left\{ 1,2,\ldots,N\right\} $ satisfies 
\[
\dfrac{\#\left\{ n\in\left\{ 1,2,\ldots,N\right\} :\delta^{\left(n\right)}\textrm{ is transcendental}\right\} }{N}\geq\frac{\mathrm{trdeg}\left(\mathcal{F}_{N}^{\left(\mathcal{D}^{\prime}\right)}\right)}{N}\geq\dfrac{\lfloor N/2\rfloor-1}{N}.
\]
Thus,
\[
\underset{N\rightarrow\infty}{\liminf}\left(\dfrac{\#\left\{ n\in\left\{ 1,2,\ldots,N\right\} :\delta^{\left(n\right)}\textrm{ is transcendental}\right\} }{N}\right)\geq\underset{N\rightarrow\infty}{\liminf}\left(\dfrac{\lfloor N/2\rfloor-1}{N}\right)\geq\dfrac{1}{2}.
\]
\end{singlespace}
\end{proof}

\section{Non-Alternating Analogues}

\subsection{Two Additional Sequences}

\begin{singlespace}
\noindent\phantom{}\medskip{}

\noindent Replacing $\left(-1\right)^{k}$ by $\left(1\right)^{k}=1$
in the numerator of the series in (8) gives the ``non-alternating
analogue'' of the sequence $\eta^{\left(n\right)}$,
\begin{equation}
\widetilde{\eta}^{\left(n\right)}=-n!{\displaystyle \sum_{k=1}^{\infty}\dfrac{1}{k^{n}\cdot k!}}.
\end{equation}
We then define the corresponding analogue of the sequence $\delta^{\left(n\right)}$
implicitly through the following natural counterpart to (9),
\begin{equation}
\gamma^{\left(n\right)}+\dfrac{\widetilde{\delta}^{\left(n\right)}}{e}=\widetilde{\eta}^{\left(n\right)},
\end{equation}

\noindent where the $\widetilde{\delta}^{\left(n\right)}$ and $\widetilde{\eta}^{\left(n\right)}$,
for $n\in\mathbb{Z}_{\geq0}$, will be called the \emph{generalized
non-alternating Euler-Gompertz} and \emph{non-alternating Eta} constants,
respectively (with $\widetilde{\delta}^{\left(1\right)}=\widetilde{\delta}=-e\textrm{Ei}\left(1\right)$
and $\widetilde{\eta}^{\left(1\right)}=\widetilde{\eta}=\textrm{Ein}\left(-1\right)$
denoting the \emph{ordinary} non-alternating Euler-Gompertz and non-alternating
Eta constants). As a counterpart of (2), we thus have
\begin{equation}
\gamma+\dfrac{\widetilde{\delta}}{e}=-{\displaystyle \sum_{k=1}^{\infty}\dfrac{1}{k\cdot k!}}=\textrm{Ein}\left(-1\right),
\end{equation}
which is shown to be transcendental by Theorem 4 of Section 4.3.\newpage{}
\end{singlespace}

\subsection{Asymptotic Behavior}

\begin{singlespace}
\noindent\phantom{}\medskip{}

\noindent Table 2 provides the first 16 values of the generalized
non-alternating Euler-Gompertz and non-alternating Eta constants.
From its third column, we can see that the (consistently negative)
$\widetilde{\eta}^{\left(n\right)}$ values are comparable in absolute
magnitude to the corresponding $\eta^{\left(n\right)}$ of Table 1,
and thus increase approximately factorially. However, the (consistently
negative) $\widetilde{\delta}^{\left(n\right)}$ values in the second
column are quite different in magnitude from the $\delta^{\left(n\right)}$
of Table 1 because they, like the $\widetilde{\eta}^{\left(n\right)}$,
grow approximately factorially as well.
\end{singlespace}

\begin{center}
Table 2. Values of $\widetilde{\delta}^{\left(n\right)}$ and $\widetilde{\eta}^{\left(n\right)}$
for $n\in\left\{ 0,1,\ldots,15\right\} $
\par\end{center}

\begin{singlespace}
\begin{center}
\begin{tabular}{|c|c|c|}
\hline 
{\footnotesize$\boldsymbol{n}$} & {\footnotesize$\boldsymbol{\widetilde{\delta}^{\left(n\right)}}$} & {\footnotesize$\boldsymbol{\widetilde{\eta}^{\left(n\right)}}$}\tabularnewline
\hline 
\hline 
{\footnotesize$0$} & {\footnotesize$-7.3890560989\ldots$} & {\footnotesize$-1.7182818284\ldots$}\tabularnewline
\hline 
{\footnotesize$1$} & {\footnotesize$-5.1514643229\ldots$} & {\footnotesize$-1.3179021514\ldots$}\tabularnewline
\hline 
{\footnotesize$2$} & {\footnotesize$-11.6100810692\ldots$} & {\footnotesize$-2.2929981450\ldots$}\tabularnewline
\hline 
{\footnotesize$3$} & {\footnotesize$-32.2422478186\ldots$} & {\footnotesize$-6.4163856531\ldots$}\tabularnewline
\hline 
{\footnotesize$4$} & {\footnotesize$-131.4700021171\ldots$} & {\footnotesize$-24.8036368256\ldots$}\tabularnewline
\hline 
{\footnotesize$5$} & {\footnotesize$-651.8492520019\ldots$} & {\footnotesize$-121.9625302861\ldots$}\tabularnewline
\hline 
{\footnotesize$6$} & {\footnotesize$-3,916.6763381264\ldots$} & {\footnotesize$-725.7973399920\ldots$}\tabularnewline
\hline 
{\footnotesize$7$} & {\footnotesize$-27,400.1009939265\ldots$} & {\footnotesize$-5,060.0849690569\ldots$}\tabularnewline
\hline 
{\footnotesize$8$} & {\footnotesize$-219,211.5495238585\ldots$} & {\footnotesize$-40,399.8007638272\ldots$}\tabularnewline
\hline 
{\footnotesize$9$} & {\footnotesize$-1,972,830.5386794810\ldots$} & {\footnotesize$-363,237.5069807080\ldots$}\tabularnewline
\hline 
{\footnotesize$10$} & {\footnotesize$-19,728,269.2785592814\ldots$} & {\footnotesize$-3,630,582.2647192272\ldots$}\tabularnewline
\hline 
{\footnotesize$11$} & {\footnotesize$-217,010,407.4543390336\ldots$} & {\footnotesize$-39,926,583.2712579089\ldots$}\tabularnewline
\hline 
{\footnotesize$12$} & {\footnotesize$-2,604,123,546.6077724786\ldots$} & {\footnotesize$-479,060,223.3022788289\ldots$}\tabularnewline
\hline 
{\footnotesize$13$} & {\footnotesize$-33,853,598,434.1585762780\ldots$} & {\footnotesize$-6,227,401,522.0546076014\ldots$}\tabularnewline
\hline 
{\footnotesize$14$} & {\footnotesize$-473,950,346,279.9399036233\ldots$} & {\footnotesize$-87,180,954,721.7674533138\ldots$}\tabularnewline
\hline 
{\footnotesize$15$} & {\footnotesize$-7,109,255,026,416.1913290094\ldots$} & {\footnotesize$-1,307,694,336,767.4617097988\ldots$}\tabularnewline
\hline 
\end{tabular}\medskip{}
\par\end{center}
\end{singlespace}

\begin{singlespace}
The following proposition quantifies the asymptotic behavior of both
sequences.
\end{singlespace}
\begin{prop}
As $n\rightarrow\infty$:

\begin{singlespace}
\noindent (i) $\widetilde{\delta}^{\left(n\right)}=-n!\left(2e+e/3^{n+1}+O\left(1/5^{n}\right)\right)$;
and

\noindent (ii) $\widetilde{\eta}^{\left(n\right)}=-n!\left(1+1/2^{n+1}+O\left(1/3^{n}\right)\right)$.
\end{singlespace}
\end{prop}
\begin{proof}
\phantom{}\medskip{}

\begin{singlespace}
\noindent See the Appendix.
\end{singlespace}
\end{proof}

\subsection{Transcendence Results}

\begin{singlespace}
\noindent\phantom{}\medskip{}

\noindent Despite the clear resemblance of (27) to (9), the former
relation does not admit a moment-decomposition interpretation analogous
to (4) because the ``positive partial moment'' component, $\widetilde{\eta}^{\left(n\right)}$,
is strictly negative. Moreover, for the case of $n=1$, there does
not appear to be anything particularly salient about the coefficient
$1/e$ in the linear combination (28) as there was for this coefficient
in (2). Thus, the similarity of (28) to (2) arises primarily from
the status of $\widetilde{\eta}^{\left(n\right)}$ as a non-alternating
analogue of $\eta^{\left(n\right)}$, with $\widetilde{\delta}^{\left(n\right)}$
emerging implicitly by constructing (27) to mimic (9).
\end{singlespace}

\begin{singlespace}
Nevertheless, it is quite straightforward to obtain results analogous
to Theorem 2, Corollary 1, and Theorem 3, as shown below.
\end{singlespace}
\begin{thm}
The set of values $\left\{ \widetilde{\eta}^{\left(n\right)}\right\} _{n\geq0}$
are algebraically independent over $\overline{\mathbb{Q}}$, implying
that $\widetilde{\eta}^{\left(n\right)}$ is transcendental for all
$n\in\mathbb{Z}_{\geq0}$.\newpage{}
\end{thm}
\begin{proof}
\phantom{}\medskip{}

\begin{singlespace}
\noindent The proof of this result is essentially the same as that
of Theorem 2 after setting $z=1$ (rather than $z=-1$) in assertion
(II) of Section 3.\footnote{For $n=1$, the present theorem proves the transcendence of $\gamma+\widetilde{\delta}/e=\textrm{Ein}\left(-1\right)$,
as noted previously. For this same value of $n$, one can prove the
transcendence of $\textrm{Ein}\left(z\right)$ for any $z\in\overline{\mathbb{Q}}^{\times}$
simply by replacing $z$ by $-z$ in assertion (II) of Section 3.} 
\end{singlespace}
\end{proof}
\begin{cor}
For all $n\in\mathbb{Z}_{\geq1}$:

\begin{singlespace}
\noindent (i) at least one element of the pair $\left\{ \gamma^{\left(n\right)},\widetilde{\delta}^{\left(n\right)}/e\right\} $
is transcendental;

\noindent (ii) at least one element of the pair $\left\{ \gamma^{\left(n\right)},\widetilde{\delta}^{\left(n\right)}\right\} $
is transcendental; and

\noindent (iii) at least two elements of the triple $\left\{ \gamma^{\left(n\right)},\widetilde{\delta}^{\left(n\right)}/e,\widetilde{\delta}^{\left(n\right)}\right\} $
are transcendental.
\end{singlespace}
\end{cor}
\begin{proof}
\phantom{}

\medskip{}
\noindent The proofs of all three parts of this corollary are analogous to those
of the corresponding parts of Corollary 1.
\end{proof}
\selectlanguage{american}%
\begin{thm}
For the sequences \foreignlanguage{english}{$\left\{ \widetilde{\delta}^{\left(n\right)}/e\right\} _{n\geq1}$
and $\left\{ \widetilde{\delta}^{\left(n\right)}\right\} _{n\geq1}$:}

\selectlanguage{english}%
\begin{singlespace}
\noindent (i) $\widetilde{\delta}^{\left(n\right)}/e$ is transcendental
infinitely often;

\noindent (ii) $\underset{N\rightarrow\infty}{\liminf}\left(\#\left\{ n\in\left\{ 1,2,\ldots,N\right\} :\widetilde{\delta}^{\left(n\right)}/e\textrm{ is transcendental}\right\} /N\right)\geq1/2$;

\noindent (iii) $\widetilde{\delta}^{\left(n\right)}$ is transcendental
infinitely often; and

\noindent (iv) $\underset{N\rightarrow\infty}{\liminf}\left(\#\left\{ n\in\left\{ 1,2,\ldots,N\right\} :\widetilde{\delta}^{\left(n\right)}\textrm{ is transcendental}\right\} /N\right)\geq1/2$.
\end{singlespace}
\end{thm}
\selectlanguage{english}%
\begin{proof}
\phantom{}\medskip{}

\begin{singlespace}
\noindent The proofs of all parts of this theorem are analogous to
those of the corresponding parts of Theorem 3.
\end{singlespace}
\end{proof}

\section{Conclusion}

\begin{singlespace}
\noindent In the present article, we defined sequences of generalized
Euler-Mascheroni, Euler-Gompertz, and Eta constants, denoted by $\gamma^{\left(n\right)}$,
$\delta^{\left(n\right)}$, and $\eta^{\left(n\right)}$, respectively.
After characterizing the basic asymptotic behavior of these sequences,
we provided the following (unconditional) results:
\end{singlespace}

\begin{singlespace}
{\footnotesize$\bullet\:$} the $\gamma^{\left(n\right)}$ (and therefore
$\Gamma^{\left(n\right)}\left(1\right)=\left(-1\right)^{n}\gamma^{\left(n\right)}$)
are transcendental infinitely often, with the density of transcendental
values among $n\in\left\{ 1,2,\ldots,N\right\} $ bounded below by
$\beta\left(N\right)=\max\left\{ 0,\sqrt{N-1}-3/2\right\} /N$, where
$\beta\left(N\right)\asymp N^{-1/2}$ as $N\rightarrow\infty$;

{\footnotesize$\bullet\:$} the $\eta^{\left(n\right)}$ are algebraically
independent, and therefore transcendental, for all $n\in\mathbb{Z}_{\geq0}$;

{\footnotesize$\bullet\:$} at least one element of each pair, $\left\{ \gamma^{\left(n\right)},\delta^{\left(n\right)}/e\right\} $
and $\left\{ \gamma^{\left(n\right)},\delta^{\left(n\right)}\right\} $,
and at least two elements of the triple $\left\{ \gamma^{\left(n\right)},\delta^{\left(n\right)}/e,\delta^{\left(n\right)}\right\} $
are transcendental for all $n\in\mathbb{Z}_{\geq1}$; and

{\footnotesize$\bullet\:$} both the $\delta^{\left(n\right)}/e$
and $\delta^{\left(n\right)}$ are transcendental infinitely often,
with lower asymptotic densities of at least $1/2$.
\end{singlespace}

Subsequently, for the generalized non-alternating Euler-Gompertz and
Eta constants, denoted by $\widetilde{\delta}^{\left(n\right)}$ and
$\widetilde{\eta}^{\left(n\right)}$, respectively, we found:

\begin{singlespace}
{\footnotesize$\bullet\:$} the $\widetilde{\eta}^{\left(n\right)}$
are algebraically independent, and therefore transcendental, for all
$n\in\mathbb{Z}_{\geq0}$;

{\footnotesize$\bullet\:$} at least one element of each pair, $\left\{ \gamma^{\left(n\right)},\widetilde{\delta}^{\left(n\right)}/e\right\} $
and $\left\{ \gamma^{\left(n\right)},\widetilde{\delta}^{\left(n\right)}\right\} $,
and at least two elements of the triple $\left\{ \gamma^{\left(n\right)},\widetilde{\delta}^{\left(n\right)}/e,\widetilde{\delta}^{\left(n\right)}\right\} $
are transcendental for all $n\in\mathbb{Z}_{\geq1}$; and

{\footnotesize$\bullet\:$} both the $\widetilde{\delta}^{\left(n\right)}/e$
and $\widetilde{\delta}^{\left(n\right)}$ are transcendental infinitely
often, with lower asymptotic densities of at least $1/2$.

Although our primary focus has been to investigate properties of the
sequences $\gamma^{\left(n\right)}$, $\delta^{\left(n\right)}$,
$\eta^{\left(n\right)}$, $\widetilde{\delta}^{\left(n\right)}$,
and $\widetilde{\eta}^{\left(n\right)}$, it is important to note
that further insights may be obtained by decomposing the sequence
of $E$-functions in (22) by the method of multisections. For example,
$F_{n}\left(z\right),\:n\in\mathbb{Z}_{\geq0}$ can be split into
the mod-2 multisection components
\[
G_{n}\left(z\right)=-n!{\displaystyle \sum_{k=1,3,\ldots}\dfrac{z^{k}}{k^{n}\cdot k!}}
\]
and
\[
H_{n}\left(z\right)=-n!{\displaystyle \sum_{k=2,4,\ldots}\dfrac{z^{k}}{k^{n}\cdot k!}},
\]
each of which is itself an $E$-function, and these new functions
can be arranged into a system of linear differential equations with
coefficients in $\mathbb{Q}\left(z\right)$ and a unique finite singularity
at $z=0$. By employing the Kolchin-Ostrowski differential-field framework
(see Srinivasan {[}16{]}), we then can show that, for any $m\in\mathbb{Z}_{\geq1}$,
the functions $G_{1}\left(z\right),H_{1}\left(z\right),G_{2}\left(z\right),H_{2}\left(z\right),\ldots,G_{m}\left(z\right),H_{m}\left(z\right)$
are algebraically independent over $\mathbb{C}\left(z,G_{0}\left(z\right),H_{0}\left(z\right)\right)$.
This allows us to conclude, via the refined Siegel-Shidlovskii theorem
(see Beukers {[}13{]}), that for $z=1$, the set of values $\left\{ G_{1}\left(1\right),H_{1}\left(1\right),G_{2}\left(1\right),H_{2}\left(1\right),\ldots,G_{m}\left(1\right),H_{m}\left(1\right)\right\} $
are algebraically independent over $\overline{\mathbb{Q}}\left(G_{0}\left(1\right),H_{0}\left(1\right)\right)$,
and therefore transcendental. Finally, recognizing that
\[
G_{n}\left(1\right)=-n!{\displaystyle \sum_{k=1,3,\ldots}\dfrac{1}{k^{n}\cdot k!}=\dfrac{1}{2}\left(\widetilde{\eta}^{\left(n\right)}-\eta^{\left(n\right)}\right)=\dfrac{1}{2}\left(\dfrac{\widetilde{\delta}^{\left(n\right)}}{e}-\dfrac{\delta^{\left(n\right)}}{e}\right)}
\]
and
\[
H_{n}\left(1\right)=-n!{\displaystyle \sum_{k=2,4,\ldots}\dfrac{1}{k^{n}\cdot k!}=\dfrac{1}{2}\left(\widetilde{\eta}^{\left(n\right)}+\eta^{\left(n\right)}\right)=\dfrac{1}{2}\left(2\gamma^{\left(n\right)}+\dfrac{\widetilde{\delta}^{\left(n\right)}}{e}+\dfrac{\delta^{\left(n\right)}}{e}\right)},
\]
we can use the fact that $\widetilde{\delta}^{\left(n\right)}/e-\delta^{\left(n\right)}/e=2G_{n}\left(1\right)$
is transcendental to conclude that at least one element of the pair
$\left\{ \delta^{\left(n\right)}/e,\widetilde{\delta}^{\left(n\right)}/e\right\} $
is transcendental for all $n\in\mathbb{Z}_{\geq1}$; and further (via
Corollary 1(i) and Corollary 2(i)) that at least two elements of the
triple $\left\{ \gamma^{\left(n\right)},\delta^{\left(n\right)}/e,\widetilde{\delta}^{\left(n\right)}/e\right\} $
are transcendental.
\end{singlespace}

\section*{Appendix}

\begin{singlespace}
\noindent\emph{Proof of Proposition 1.\medskip{}
}

\noindent We consider the three parts of the proposition in the order:
(iii), (ii), (i).\medskip{}

\noindent (iii) From (8), it is easy to see that
\[
\eta^{\left(n\right)}<n!\left[1-\dfrac{1}{2^{n+1}}+\left(\dfrac{1}{6}\right)\dfrac{1}{3^{n}}+\left(\dfrac{1}{24}\right)\dfrac{1}{4^{n}}+\cdots\right]
\]
\[
<n!\left[1-\dfrac{1}{2^{n+1}}+\left(\dfrac{1}{6}+\dfrac{1}{24}+\cdots\right)\dfrac{1}{3^{n}}\right]
\]
\[
=n!\left[1-\dfrac{1}{2^{n+1}}+\dfrac{\left(e-5/2\right)}{3^{n}}\right]
\]
and
\[
\eta^{\left(n\right)}>n!\left[1-\dfrac{1}{2^{n+1}}-\left(\dfrac{1}{6}\right)\dfrac{1}{3^{n}}-\left(\dfrac{1}{24}\right)\dfrac{1}{4^{n}}-\cdots\right]
\]
\[
>n!\left[1-\dfrac{1}{2^{n+1}}-\left(\dfrac{1}{6}+\dfrac{1}{24}+\cdots\right)\dfrac{1}{3^{n}}\right]
\]
\[
=n!\left[1-\dfrac{1}{2^{n+1}}-\dfrac{\left(e-5/2\right)}{3^{n}}\right].
\]
It then follows that
\[
\eta^{\left(n\right)}=n!\left(1-\dfrac{1}{2^{n+1}}+O\left(\dfrac{1}{3^{n}}\right)\right).
\]

\noindent (ii) Substituting $u=e^{-x}$ into the integral in (6) yields
\[
\delta^{\left(n\right)}=\left(-1\right)^{n+1}e{\displaystyle \int_{1}^{\infty}\left[\ln\left(u\right)\right]^{n}e^{-u}du}
\]
\[
\Longrightarrow\left|\delta^{\left(n\right)}\right|=e{\displaystyle \int_{1}^{\infty}\left[\ln\left(u\right)\right]^{n}e^{-u}du}=e{\displaystyle \int_{1}^{\infty}e^{\phi_{n}\left(u\right)}du},\qquad\qquad\textrm{(A1)}
\]
where $\phi_{n}\left(u\right)=n\ln\left(\ln\left(u\right)\right)-u$.
We then employ Laplace's (saddle-point) method to approximate this
integral.
\end{singlespace}

\begin{singlespace}
Taking derivatives of $\phi_{n}\left(u\right)$ with respect to $u$
gives
\[
\phi_{n}^{\prime}\left(u\right)=\dfrac{n}{u\ln\left(u\right)}-1\:\textrm{and }\:\phi_{n}^{\prime\prime}\left(u\right)=-\dfrac{n}{u^{2}\ln\left(u\right)}\left(1+\dfrac{1}{\ln\left(u\right)}\right)<0,
\]
revealing that $\phi_{n}\left(u\right)$ enjoys a unique global maximum
at the saddle point
\[
u^{*}\ln\left(u^{*}\right)=n\Longleftrightarrow u^{*}=\dfrac{n}{W\left(n\right)},\:\ln\left(u^{*}\right)=W\left(n\right).
\]
Then
\[
\phi_{n}\left(u^{*}\right)=n\ln\left(\ln\left(u^{*}\right)\right)-\dfrac{n}{\ln\left(u^{*}\right)}=n\ln\left(W\left(n\right)\right)-\dfrac{n}{W\left(n\right)},
\]
\[
\left|\phi_{n}^{\prime\prime}\left(u^{*}\right)\right|=\dfrac{W\left(n\right)}{n}\left(1+\dfrac{1}{W\left(n\right)}\right)=\dfrac{W\left(n\right)+1}{n},
\]
and (A1) can be approximated by
\[
\left|\delta^{\left(n\right)}\right|=e\cdot e^{\phi_{n}\left(u^{*}\right)}\sqrt{\dfrac{2\pi}{\left|\phi_{n}^{\prime\prime}\left(u^{*}\right)\right|}}\left(1+o\left(1\right)\right)
\]
\[
=e\left[W\left(n\right)\right]^{n}\exp\left(-\dfrac{n}{W\left(n\right)}\right)\sqrt{\dfrac{2\pi n}{W\left(n\right)+1}}\left(1+o\left(1\right)\right).
\]

\end{singlespace}

\begin{singlespace}
\noindent (i) Finally, we assemble the results in parts (iii) and
(ii) via (9), giving
\[
\gamma^{\left(n\right)}=n!\left(1-\dfrac{1}{2^{n}\cdot2!}+\dfrac{1}{3^{n}\cdot3!}-\dfrac{1}{4^{n}\cdot4!}+\cdots\right)-\dfrac{1}{e}\left(-1\right)^{n+1}e{\displaystyle \int_{1}^{\infty}\left[\ln\left(u\right)\right]^{n}e^{-u}du}
\]
\[
=n!\left(1-\dfrac{1}{2^{n+1}}+O\left(\dfrac{1}{3^{n}}\right)\right)-\left(-1\right)^{n+1}\left[W\left(n\right)\right]^{n}\exp\left(-\dfrac{n}{W\left(n\right)}\right)\sqrt{\dfrac{2\pi n}{W\left(n\right)+1}}\left(1+o\left(1\right)\right)
\]
\[
=n!\left(1-\dfrac{1}{2^{n+1}}+O\left(\dfrac{1}{3^{n}}\right)\right)+\left(n!\right)o\left(\dfrac{1}{3^{n}}\right)
\]
\[
=n!\left(1-\dfrac{1}{2^{n+1}}+O\left(\dfrac{1}{3^{n}}\right)\right).
\]
\medskip{}

\noindent\emph{Proof of Proposition 2.}\medskip{}

\noindent We consider the two parts of the proposition in reverse
order.\medskip{}

\noindent (ii) From (26), it follows that
\[
\widetilde{\eta}^{\left(n\right)}>-n!\left[1+\dfrac{1}{2^{n+1}}+\left(\dfrac{1}{6}+\dfrac{1}{24}+\cdots\right)\dfrac{1}{3^{n}}\right]
\]
\[
=-n!\left[1+\dfrac{1}{2^{n+1}}+\dfrac{\left(e-5/2\right)}{3^{n}}\right]
\]
and
\[
\widetilde{\eta}^{\left(n\right)}<-n!\left[1+\dfrac{1}{2^{n+1}}+\left(\dfrac{1}{6}\right)\dfrac{1}{3^{n}}\right],
\]
implying
\[
\widetilde{\eta}^{\left(n\right)}=-n!\left(1+\dfrac{1}{2^{n+1}}+O\left(\dfrac{1}{3^{n}}\right)\right).
\]

\end{singlespace}

(i) Now solve for $\widetilde{\delta}^{\left(n\right)}$ via (27),
using the initial expression for $\gamma^{\left(n\right)}$ in the
proof of Proposition 1(i): 
\[
\begin{array}{c}
\widetilde{\delta}^{\left(n\right)}=e\left[-n!\left(1+\dfrac{1}{2^{n}\cdot2!}+\dfrac{1}{3^{n}\cdot3!}+\dfrac{1}{4^{n}\cdot4!}+\cdots\right)-n!\left(1-\dfrac{1}{2^{n}\cdot2!}+\dfrac{1}{3^{n}\cdot3!}-\dfrac{1}{4^{n}\cdot4!}+\cdots\right)\right.\\
\left.+\dfrac{1}{e}\left(-1\right)^{n+1}e{\displaystyle \int_{1}^{\infty}\left[\ln\left(u\right)\right]^{n}e^{-u}du}\right]
\end{array}
\]
\[
=e\left[-2n!\left(1+\dfrac{1}{2\cdot3^{n+1}}+O\left(\dfrac{1}{5^{n}}\right)\right)+\left(n!\right)o\left(\dfrac{1}{5^{n}}\right)\right]
\]
\[
=-n!\left(2e+\dfrac{e}{3^{n+1}}+O\left(\dfrac{1}{5^{n}}\right)\right).
\]

\end{document}